\documentclass[12pt]{article}

\usepackage{latexsym}
\usepackage{amssymb}
\usepackage{amsmath}
\usepackage{enumerate}
\usepackage{amsmath,amsthm,amsfonts,amssymb,latexsym,xy}
\xyoption{all}
\usepackage{enumerate}

\newcommand{\norm}[1]{\lVert#1 \rVert}
\newcommand{\abs}[1]{\lvert#1 \rvert}

\DeclareMathOperator{\spr}{spr}

\newtheorem{theorem}{Theorem}
\newtheorem{lemma}[theorem]{Lemma}

\newtheorem{definition}[theorem]{Definition}
\newtheorem{example}[theorem]{Example}
\newtheorem{remark}[theorem]{Remark}

\begin{document}
\title{Matrix Ordered Operator Algebras.}
\author{Ekaterina Juschenko, Stanislav Popovych
\\\vspace{0.1cm}\\
{\footnotesize \sl Department of Mathematics, Chalmers University of
Technology,
SE-412 96 G\"oteborg, Sweden}\\
{\footnotesize \sl jushenko\symbol{64}math.chalmers.se} \\
{\footnotesize \sl stanislp\symbol{64}math.chalmers.se} }
\date{}
\maketitle

 \footnotetext{ 2000 {\it Mathematics Subject
Classification}: 46L05,  46L07 (Primary) 47L55, 47L07, 47L30
(Secondary) }

\begin{abstract}
We study the question when for a given $*$-algebra $\mathcal{A}$ a
sequence of cones $C_n\in M_n(\mathcal{A})$  can be realized as
cones of positive operators in a faithful $*$-representation of
$\mathcal{A}$ on a Hilbert space. A characterization of operator
algebras which are completely boundedly isomorphic to
$C\sp*$-algebras is presented.

\medskip\par\noindent
KEYWORDS:  $*$-algebra,  operator algebra, $C\sp*$-algebra,
completely bounded homomorphism, Kadison's problem.
\end{abstract}

\section{Introduction}

Effros and Choi gave in~\cite{ChoiEffros} an abstract
characterization of the self-adjoint subspaces $S$ in $C^*$-algebras
with hierarchy of cones of positive elements in $M_n(S)$. In Section
2 of the present paper we are concerned with the same question for
$*$-subalgebras of $C\sp*$-algebras. More precisely, let
$\mathcal{A}$ be an associative $*$-algebra  with unit. In
Theorem~\ref{operatorstar} we present a characterization of the
collections of cones $C_n\subseteq M_n(\mathcal{A})$ for which there
exist faithful $*$-representation $\pi$ of $\mathcal{A}$ on a
Hilbert space $H$ such that $C_n$ coincides with the cone of
positive operators contained in $\pi^{(n)}(M_n(\mathcal{A}))$. Here
$\pi^{(n)}( (x_{i,j}) ) = ( \pi(x_{i,j}) )$ for every matrix
$(x_{i,j})\in M_n(\mathcal{A})$. Note that we do not assume that
$\mathcal{A}$ has any faithful $*$-representation. This follows from
the requirements imposed on the cones. In terms close to Effros and
Choi we give an abstract characterizations of matrix ordered (not
necessary closed) operator $*$-algebras up to complete order
$*$-isomorphism.

Based on this characterization we study the question when an
operator algebra is similar to a $C^*$-algebra.

Let $\mathcal{B}$ be a unital (closed) operator algebra in $B(H)$.
In~\cite{merdy} C. Le Merdy  presented necessary and sufficient
conditions for $\mathcal{B}$ to be self-adjoint. These conditions
involve all completely isometric representations of $\mathcal{B}$ on
Hilbert spaces.  Our characterization  is different in the following
respect.  If $S$ is a bounded invertible operator in $B(H)$ and
$\mathcal{A}$ is a $C\sp*$-algebra in $B(H)$ then the operator
algebra $S^{-1} \mathcal{A} S$ is not necessarily self-adjoint  but
only isomorphic to a $C^*$-algebra via completely bounded
isomorphism with completely bounded inverse.  By Haagerup's theorem
every completely bounded isomorphism $\pi$ from a $C\sp*$-algebra
$\mathcal{A}$ to an operator algebra $\mathcal{B}$ has the form
$\pi(a)=S^{-1}\rho(a)S$, $a\in \mathcal{A}$, for some
$*$-isomorphism $\rho:\mathcal{A}\rightarrow B(H)$ and invertible
$S\in B(H)$. Thus the question whether  an operator algebra
$\mathcal{B}$ is completely boundedly isomorphic to a $C^*$-algebra
via isomorphism which has completely bounded inverse, is equivalent
to the question if there is bounded invertible operator $S$ such
that $S \mathcal{B} S^{-1}$ is a $C^*$-algebra.

We will present a criterion for an operator algebra $\mathcal{B}$ to
be completely boundedly isomorphic to a $C\sp*$-algebra in terms of
the existence of a collection of cones $C_n\in M_n(\mathcal{B})$
satisfying certain axioms (see def.~\ref{consdef}). The axioms are
derived from the properties of the cones of positive elements of a
$C^*$-algebra preserved under completely bounded isomorphisms.

The main results are contained in section 2. We define a
$*$-admissible sequence of cones in an operator algebra and present
a criterion in Theorem~\ref{main} for an operator algebra to be
completely boundedly isomorphic to a $C^*$-algebra.

In the last section we consider the operator algebras and
collections of cones associated with Kadison similarity problem.

\section{Operator realizations of matrix-ordered $*$-algebras. }
The aim of this section is to give necessary and sufficient
conditions on a sequences of cones $C_n\subseteq
M_n(\mathcal{A})_{sa}$ for a unital $*$-algebra $\mathcal{A}$ such
that $C_n$ coincides with the cone $M_n(\mathcal{A})\cap
M_n(B(H))^+$ for some realization of $\mathcal{A}$ as a
$*$-subalgebra of $B(H)$, where $M_n(B(H))^+$ denotes the set of
positive operators acting on $H^n = H\oplus \ldots \oplus H$.

In \cite{Popovych} it was proved that a $*$-algebra $\mathcal{A}$
with unit $e$ is a $*$-subalgebra of $B(H)$ if and only if  there
is an algebraically admissible cone on $\mathcal{A}$ such that $e$
is an Archimedean order unit. Applying this result to some
inductive limit of $M_{2^n}(\mathcal{A})$ we obtain the desired
characterization in Theorem \ref{operatorstar}.

First we give necessary definitions and fix  notations. Let
$\mathcal{A}_{sa}$ denote the set of self-adjoint elements in
$\mathcal{A}$. A subset $C\subset \mathcal{A}_{sa}$ containing unit
$e$ of $\mathcal{A}$ is \textit{algebraically admissible} cone (see
\cite{powers}) provided that
\begin{enumerate}[(i)]
\item   $C$ is a cone  in $\mathcal{A}_{sa}$, i.e. $\lambda
x+\beta y\in C$ for all $x$, $y\in C$ and $\lambda\geq 0$,
$\beta\geq 0$, $\lambda,\beta\in\mathbb{R}$; \item $C \cap (-C)
=\{0\}$; \item $x C x^* \subseteq C$ for every $x\in \mathcal{A}$;
\end{enumerate}

We call $e\in \mathcal{A}_{sa}$ an \textit{order unit} if for
every $x\in \mathcal{A}_{sa}$ there exists $r>0$ such that
$re+x\in C$. An order unit $e$ is \textit{Archimedean} if $re+x\in
C$ for all $r>0$ implies that $x\in C$

In what follows we will need the following.
\begin{theorem}\label{onecone}
Let $\mathcal{A}$ be a  $*$-algebra with unit $e$ and $C\subseteq
\mathcal{A}_{sa}$ be a cone containing $e$. If $x C x^* \subseteq
C$ for every $x\in \mathcal{A}$ and $e$ is an Archimedean order
unit then there is a  unital $*$-representation $\pi: \mathcal{A}
\to B(H)$ such that $\pi(C) = \pi(\mathcal{A}_{sa})\cap B(H)^+$.
Moreover
\begin{enumerate} \item\label{onecone1} $\norm{\pi(x)} = \inf \{r>0 : r^2 \pm x^*x\in C\}$.
\item\label{onecone2} $\ker \pi =\{ x : x^*x \in C \cap (-C)\}$.
\item\label{onecone3} If $C \cap (-C) = \{0\}$ then $\ker \pi =\{
0\}$, $\norm{\pi(a)} = \inf \{r>0 : r \pm a\in C\}$ for all
$a=a^*\in \mathcal{A}$ and $\pi(C) = \pi(\mathcal{A})\cap B(H)^+$
\end{enumerate}
\end{theorem}
\begin{proof}
Following the same lines as in \cite{Popovych} one obtains that the
function $\| \cdot\|: \mathcal{A}_{sa}\to \mathbb{R}_+$ defined as
$$ \| a \| = \inf \{ r> 0: re \pm a \in C \} $$
is a seminorm on $\mathbb{R}$-space  $\mathcal{A}_{sa}$ and $|x| =
\sqrt{\|x^*x\|}$ for $x\in \mathcal{A}$ defines a pre-$C\sp*$-norm
on $\mathcal{A}$. If $N$ denote the null-space of  $|\cdot|$ then
the completion $\mathcal{B} = \overline{\mathcal{A}/ N}$ with
respect to this norm is a $C\sp*$-algebra and canonical epimorphism
$\pi: \mathcal{A} \to \mathcal{A}/ N$ extends to a unital
$*$-homomorphism $\pi: \mathcal{A}\to \mathcal{B}$. We can assume
without loss of generality that $\mathcal{B}$ is a concrete
$C\sp*$-algebra in  $B(H)$ for some Hilbert space $H$. Thus
$\pi:\mathcal{A} \to B(H)$ can be regarded as a unital
$*$-representation. Clearly,
 $$\|\pi(x)\|= | x |\text{ for all } x\in \mathcal{A}.$$
This implies~\ref{onecone1}.

To show \ref{onecone2} take $x\in \ker\pi$ then $\|\pi(x)\|=0$ and
$re\pm x^*x\in C$ for all $r>0$. Since $e$ is an Archimedean unit we
have $x^*x\in C\cap(-C)$. Conversely if  $x^*x\in C\cap(-C)$ then
$re\pm x^*x\in C$, for all $r>0$, hence $\|\pi(x)\|=0$ and
\ref{onecone2} holds.

Let us prove that $\pi(C) = \pi(\mathcal{A}_{sa})\cap B(H)^+$. Let
$x\in \mathcal{A}_{sa}$ and $\pi(x)\geq 0$. Then there exists a
constant $\lambda>0$ such that $\|\lambda I_{H}-\pi(x)\|\leq
\lambda$, hence $|\lambda e-x|\leq \lambda$. Since $\|a\|\leq |a|$
for all self-adjoint $a\in \mathcal{A}$, see Lemma 3.3 of
\cite{Popovych}, we have $\|\lambda e-x\|\leq \lambda$. Thus given
$\varepsilon>0$ we have $(\lambda+\varepsilon)e\pm (\lambda e-x)\in
C$. Hence $\varepsilon e+x\in {C}$. Since $e$ is Archimedean $x\in
{C}$.

Conversely, let $x\in {C}$. To show that $\pi(x)\geq 0$ it is
sufficient to find $\lambda>0$ such that $\|\lambda
I_{H}-\pi(x)\|\leq \lambda$. Since $\|\lambda
I_{H}-\pi(x)\|=|\lambda e-x|$ we will prove that $|\lambda
e-x|\leq \lambda$ for some $\lambda>0$. From the definition of
norm $|\cdot|$ we have the following equivalences:
\begin{eqnarray}|\lambda e-x|\leq \lambda
&\Leftrightarrow& (\lambda+\varepsilon)^2e-(\lambda e-x)^2\in
C\text{ for all }\varepsilon>0\\\label{ineq1} &\Leftrightarrow&
\varepsilon_1 e+x(2\lambda e-x)\geq 0,\text{ for all
}\varepsilon_1>0.
\end{eqnarray}

By condition (iii) in the definition  of algebraically admissible
cone we have that $xyx\in C$ and $yxy\in C$ for every $x,y\in C$. If
$xy=yx$ then $xy(x+y)\in C$. Since $e$ is an order unit we can
choose $r>0$ such that $r e-x\in C$. Put $y=r e-x$ to obtain $rx(r e
-x)\in C$. Hence (\ref{ineq1}) is satisfied with
$\lambda=\frac{r}{2}$. Thus $\|\lambda e-\pi(x)\|\leq \lambda$ and
$\pi(x)\geq 0$, which proves $\pi(C) = \pi(\mathcal{A}_{sa})\cap
B(H)^+$.

In particular, for $a=a^*$ we have
\begin{gather}\label{form}
\norm{\pi(a)} = \inf \{r>0 : r I_H \pm \pi(a) \in \pi(C)\}.
\end{gather}

We now in a position to prove \ref{onecone3}. Suppose that $C\cap
(-C) = 0$. Then $\ker \pi$ is a $*$-ideal and $\ker \pi \not= 0$
implies that there exists a self-adjoint $0\not= a\in \ker\pi$, i.e.
$\abs{a} = 0$.  Inequality $\| a \| \le |a|$ implies $r e \pm a \in
C$ for all $r>0$. Since $e$ is Archimedean, $\pm a \in C$, i.e.
$a\in C\cap (-C)$ and, consequently, $a = 0$.

Since $\ker \pi = 0$ the inclusion  $r I_H \pm \pi(a) \in \pi(C)$
is equivalent to $r e \pm a \in C$, and by~(\ref{form}),
$\norm{\pi(a)} = \inf \{r>0 : r e \pm a \in C\}$. Moreover if
$\pi(a) = \pi(a)^*$ then $a = a^*$. Thus we have $\pi(C) =
\pi(A)\cap B(H)^+$.
\end{proof}

We say that a $*$-algebra $\mathcal{A}$ with unit $e$ is a
\textit{matrix ordered} if the following conditions hold:
\begin{enumerate}[(a)]
    \item for each $n\geq 1$ we are given a cone $C_n$ in $M_n(\mathcal{A})_{sa}$ and
    $e\in C_1$,
    \item $C_n\cap(-C_n)=\{0\}$ for all $n$,
    \item for all $n$ and $m$ and all $A\in M_{n\times m}(\mathcal{A})$, we
    have that $A^*C_n A \subseteq C_m$,
\end{enumerate}

We call $e\in \mathcal{A}_{sa}$ a \textit{matrix order unit}
provided that  for every $n\in \mathbb{N}$ and every $x\in
M_n(\mathcal{A})_{sa}$ there exists $r>0$ such that $re_n+x\in C_n$,
where $e_n=e\otimes I_n$. A matrix order unit is called
\textit{Archimedean matrix order unit} provided that for all $n\in
\mathbb{N}$ inclusion $re_n+x\in C_n$ for all $r>0$ implies that
$x\in C_n$.

Let $\pi:\mathcal{A}\rightarrow B(H)$ be a $*$-representation.
Define $\pi^{(n)}:M_{n}(\mathcal{A})\rightarrow M_n(B(H))$ by
$\pi^{(n)}((a_{ij}))=(\pi(a_{ij}))$.

\begin{theorem}\label{operatorstar} If $\mathcal{A}$ is a matrix-ordered $*$-algebra
with a unit e which is Archime\-dean matrix order unit then there
exists a Hilbert space $H$ and a faithful unital $*$-representation
$\tau:\mathcal{A}\rightarrow B(H)$, such that
$\tau^{(n)}(C_n)=M_n(\tau(\mathcal{A}))^+$ for all $n$. Conversely,
every unital $*$-subalgebra $\mathcal{D}$ of $B(H)$ is
matrix-ordered by cones $M_n(\mathcal{D})^+=M_n(\mathcal{D})\cap
B(H)^+$ and the unit of this algebra is an Archimedean order unit.
\end{theorem}
\begin{proof}
Consider an inductive system of $*$-algebras and  unital injective
$*$-homomorphisms:
$$\phi_n:M_{2^n}(\mathcal{A})\rightarrow M_{2^{n+1}}(\mathcal{A}),\quad \phi_n(a)=\left(%
\begin{array}{cc}
  a & 0 \\
  0 & a \\
\end{array}%
\right)\text{ for all } n\geq 0, a\in M_{2^n}(\mathcal{A}).$$ Let
$\mathcal{B}=\underrightarrow{\lim} M_{2^n}(\mathcal{A})$ be the
inductive limit of this system. By $(c)$ in the definition of the
matrix ordered algebra  we have $\phi_n(C_{2^n})\subseteq
C_{2^{n+1}}$. We will identify $M_{2^n}(\mathcal{A})$ with a
subalgebra of $\mathcal{B}$ via canonical inclusions. Let
$C=\bigcup\limits_{n\ge 1} C_{2^n}\subseteq\mathcal{B}_{sa}$ and let
$e_{\infty}$ be the unit of $\mathcal{B}$.

Let us prove that $C$ is an algebraically admissible cone. Clearly,
$C$ satisfies conditions (i) and (ii) of definition of algebraically
admissible cone. To prove (iii) suppose that $x\in \mathcal{B}$ and
$a\in C$, then for sufficiently  large $n$ we have $a\in C_{2^n}$
and $x\in M_{2^n}(\mathcal{A})$. Therefore, by $(c)$, $x^*ax\in C$.
Thus (iii) is proved.  Since $e$ is an Archimedean matrix order unit
we obviously have that $e_{\infty}$ is also an Archimedean order
unit. Thus $*$-algebra $\mathcal{B}$ satisfies assumptions of
Theorem~\ref{onecone} and there is a faithful $*$-representation
$\pi:\mathcal{B}\rightarrow B(H)$ such that $\pi(C) =
\pi(\mathcal{B}) \cap B(H)^+$.

Let $\xi_n:M_{2^n}(\mathcal{A})\rightarrow \mathcal{B}$ be canonical
injections ($n\geq 0$). Then
$\tau=\pi\circ\xi_0:\mathcal{A}\rightarrow B(H)$ is an injective
$*$-homomorphism.

We claim that $\tau^{(2^n)}$ is unitary equivalent to $\pi\circ
\xi_n$. By replacing $\pi$ with $\pi^{\alpha}$, where $\alpha$ is an
infinite cardinal, we can assume that $\pi^{\alpha}$ is unitary
equivalent to $\pi$.  Since  $\pi\circ
\xi_n:M_{2^n}(\mathcal{A})\rightarrow B(H)$ is a $*$-homomorphism
there exist unique Hilbert space $K_n$, $*$-homomorphism
$\rho_n:\mathcal{A}\rightarrow B(K_n)$ and unitary operator
$U_n:K_n\otimes \mathbb{C}^{2^n}\rightarrow H$ such that
$$\pi\circ \xi_n=U_n(\rho_n\otimes id_{M_{2^n}})U_n^*.$$
For $a\in \mathcal{A}$, we have
\begin{eqnarray*} \pi\circ \xi_0(a)&=&\pi\circ\xi_n(a\otimes
E_{2^n})\\
&=&U_n(\rho_n(a)\otimes E_{2^n})U_n^*,
\end{eqnarray*}
where $E_{2^n}$ is the identity matrix in $M_{2^n}(\mathbb{C})$.
Thus $\tau(a)=U_0\rho_0(a)U_0^*=U_n(\rho_n(a)\otimes E_{2^n})U_n^*$.
Let $\sim$ stands for the unitary equivalence of representations.
Since $\pi\circ \xi_n\sim\rho_n\otimes id_{M_{2^n}}$ and
$\pi^{\alpha}\sim \pi$ we have that $\rho_n^{\alpha}\otimes
id_{M_{2^n}}\sim\pi^{\alpha}\circ\xi_n\sim \rho_n\otimes
id_{M_{2^n}}$. Hence $\rho^\alpha_n\sim\rho_n$. Thus $\rho_n\otimes
E_{2^n}\sim \rho_n^{2^n\alpha}\sim\rho_n$. Consequently
$\rho_0\sim\rho_n$ and $\pi\circ\xi_n\sim\rho_0\otimes
id_{M_{2^n}}\sim \tau\otimes id_{M_{2^n}}$. Therefore
$\tau^{(2^n)}=\tau\otimes id_{M_{2^n}}$ is unitary equivalent to
$\pi\circ \xi_n$.

What is left to show is that
$\tau^{(n)}(C_n)=M_{n}(\tau(\mathcal{A}))^+$. Note that $\pi\circ
\xi_n(M_{2^n}(\mathcal{A}))\cap B(H)^+=\pi(C_{2^n})$. Indeed, the
inclusion $\pi\circ\xi(C_{2^n})\subseteq M_{2^n}(\mathcal{A})\cap
B(H)^+$ is obvious. To show the converse take $x\in
M_{2^n}(\mathcal{A})$ such that $\pi(x)\geq 0$. Then $x\in C\cap
M_{2^n}(\mathcal{A})$. Using $(c)$ one can easily show that $C\cap
M_{2^n}(\mathcal{A})= C_{2^n}$. Hence $\pi\circ
\xi_n(M_{2^n}(\mathcal{A}))\cap B(H)^+=\pi(C_{2^n})$. Since
$\tau^{(2^n)}$ is unitary equivalent to $\pi\circ \xi_n$ we have
that $\tau^{(2^n)}(C_{2^n})=M_{2^n}(\tau(\mathcal{A}))\cap
B(H^{2^n})^+$.

Let now show that $\tau^{(n)}(C_n)=M_{n}(\tau(\mathcal{A}))^+$.
For $X\in M_n(\mathcal{A})$ denote
$$\widetilde{X}=\left(%
\begin{array}{cc}
  X & 0_{n\times (2^n-n)} \\
  0_{(2^n-n) \times n} & 0_{(2^n-n)\times (2^n-n)} \\
\end{array}%
\right)\in M_{2^n}(\mathcal{A}).$$ Then, clearly,
$\tau^{(n)}(X)\geq 0$ if and only if
$\tau^{(2^n)}(\widetilde{X})\geq 0$. Thus $\tau^{(n)}(X)\geq 0$ is
equivalent to $\widetilde{X}\in C_{2^n}$ which in turn is
equivalent to $X\in C_n$ by $(c)$.
\end{proof}

\section{Operator Algebras completely boundedly isomorphic to $C^*$-algebras.}
The algebra $M_{n}(B(H))$ of $n\times n$ matrices with entries in
$B(H)$ has a norm $\|\cdot\|_{n}$ via the identification of
$M_n(B(H))$ with $B(H^n)$, where $H^n$ is the direct sum of $n$
copies of a Hilbert space $H$. If $\mathcal{A}$ is a subalgebra of
$B(H)$ then $M_n(\mathcal{A})$ inherits a norm $\|\cdot\|_n$ via
natural inclusion into $M_n(B(H))$. The  norms $ \|\cdot\|_n $  are
called matrix norms on the operator algebra $\mathcal{A}$. In the
sequel all operator algebras will be assumed to be norm closed.

Operator algebras $\mathcal{A}$ and $\mathcal{B}$ are called
completely boundedly isomorphic if there is a completely bounded
isomorphism $\tau:\mathcal{A}\rightarrow \mathcal{B}$ with
completely bounded inverse. The aim of this section is to give
necessary and sufficient conditions for an operator algebra to be
completely boundedly isomorphic to a $C^*$-algebra. To do this we
introduce a concept of $*$-admissible cones which reflect the
properties of the cones of positive elements of a $C^*$-algebra
preserved under completely bounded isomorphism.
\begin{definition}\label{consdef}
Let $\mathcal{B}$ be an operator algebra with unit $e$. A sequence
$C_n\subseteq M_n(\mathcal{B})$ of closed (in the norm
$\|\cdot\|_n$) cones will be called \textit{$*$-admissible} if it
satisfies the following conditions:
\begin{enumerate}
\item $e\in C_1$;  \item
\begin{enumerate}[(i)]\item
$M_n(\mathcal{B})=(C_n-C_n)+i(C_n-C_n)$, for all $n\in \mathbb{N}$,
\item $C_n\cap (-C_n)=\{0\}$, for all $n\in \mathbb{N}$, \item
$(C_n-C_n)\cap i(C_n-C_n)=\{0\}$, for all $n\in \mathbb{N}$;
\end{enumerate}
\item \begin{enumerate}[(i)]\item for all $c_1$, $c_2\in C_n$ and
$c\in C_n$, we have that $(c_1-c_2)c(c_1-c_2)\in C_n$, \item for all
$n$, $m$ and $B\in M_{n\times m}(\mathbb{C})$ we have that $B^*C_n
B\subseteq C_m$;
\end{enumerate}
\item  there is $r>0$ such that for every positive integer $n$ and  $c\in C_{n}-C_{n}$ we have
$r \|c\|e_n +c \in C_n$, \item there exists a constant $K>0$ such
that for all $n\in\mathbb{N}$ and $a$, $b\in C_n-C_n$ we have
$\|a\|_n\leq K\cdot\|a+ib\|_n$.
\end{enumerate}
\end{definition}

\begin{theorem}\label{main}
If an operator algebra $\mathcal{B}$ has a $*$-admissible sequence
of cones then there is a completely bounded isomorphism $\tau$ from
$\mathcal{B}$ onto a $C^*$-algebra $\mathcal{A}$.  If, in addition,
one of the following conditions holds
\begin{enumerate}[(1)]
\item\label{cond1} there exists $r>0$ such that  for every $n\ge 1$ and  $c, d \in C_{n}$
we have    $ \|c + d \| \ge r \|c\|$.
\item $\| (x- i y) (x+i y)\| \ge \alpha \| x - i y  \| \| x+i y  \|$
for all $x, y \in C_n - C_n$
\end{enumerate}
then the inverse $\tau^{-1}: \mathcal{A} \to \mathcal{B}$ is also
completely bounded.

Conversely, if such isomorphism $\tau$ exists then $\mathcal{B}$
possesses a $*$-admissible sequence of cones and conditions $(1)$
and $(2)$ are satisfied.
\end{theorem}

The proof will be divided into 4 lemmas.\\

Let $\{C_n\}_{n\geq 1}$ be a $*$-admissible sequence of cones of
$\mathcal{B}$.  Let $\mathcal{B}_{2^n}=M_{2^n}(\mathcal{B})$,
$\phi_n:\mathcal{B}_{2^n}\rightarrow \mathcal{B}_{2^{n+1}}$ be
unital homomorphisms given by $\phi_n(x)=\left(%
\begin{array}{cc}
  x & 0 \\
  0 & x \\
\end{array}%
\right)$, $x\in \mathcal{B}_{2^n}$. Denote by
$\mathcal{B}_{\infty}=\underrightarrow{\lim}\mathcal{B}_{2^n}$ the
 inductive limit of the system $(\mathcal{B}_{2^n},\phi_n)$. As
 all inclusions $\phi_n$ are unital $\mathcal{B}_{\infty}$ has a
 unit, denoted by $e_{\infty}$. Since $\mathcal{B}_{\infty}$ can be considered as a subalgebra of
a $C^*$-algebra of the corresponding inductive limit of
$M_{2^n}(B(H))$ we can define the closure  of $\mathcal{B}_{\infty}$
in this $C^*$-algebra denoted by $\overline{\mathcal{B}}_{\infty}$.

 Now we will define an involution on
$\mathcal{B}_{\infty}$. Let $\xi_n : M_{2^n} ( \mathcal{B} )\to
\mathcal{B}_{\infty}$ be the canonical morphisms. By $(3ii)$,
$\phi_n(C_{2^n})\subseteq C_{2^{n+1}}$. Hence $C= \bigcup\limits_{n}
\xi_n( C_{2^n})$ is a well defined cone in $\mathcal{B}_{\infty}$.
Denote by $\overline{C}$ its completion. By $(2i)$ and $(2iii)$, for
every $x\in \mathcal{B}_{2^n}$, we have $x=x_1+ix_2$ with
unique $x_1$, $x_2\in C_{2^n}-C_{2^n}$. By $(3ii)$ we have $\left(%
\begin{array}{cc}
  x_i & 0 \\
  0 & x_i \\
\end{array}%
\right)\in C_{2^{n+1}}-C_{2^{n+1}}$, $i=1,2$. Thus for every $x\in
B_{\infty}$ we have unique decomposition $x=x_1+ix_2$, $x_1\in C-C$,
$x_2\in C-C$. Hence the mapping $x\mapsto x^{\sharp}=x_1-ix_2$ is a
well defined involution on $\mathcal{B}_{\infty}$. In particular, we
have an involution on $\mathcal{B}$ which depends only on the cone
$C_1$.
\begin{lemma}\label{involution} Involution on $\mathcal{B}_{\infty}$
is defined by the involution on $\mathcal{B}$, i.e. for all
$A=(a_{ij})_{i,j}\in M_{2^n}(\mathcal{B})$
$$A^{\sharp} = (a_{ji}^{\sharp})_{i,j}.$$
\end{lemma}
\begin{proof}
Assignment $A^\circ = (a_{ji}^{\sharp})_{i,j}$, clearly, defines an
involution on $M_{2^n}(\mathcal{B})$. We need to prove that
$A^{\sharp}=A^{\circ}$.

Let $A=(a_{ij})_{i,j}\in M_{2^n}(\mathcal{B})$ be self-adjoint
$A^{\circ}=A$. Then $A=\sum\limits_{i}a_{ii}\otimes
E_{ii}+\sum\limits_{i< j}(a_{ij}\otimes
E_{ij}+a_{ij}^{\sharp}\otimes E_{ji})$ and $a_{ii}^{\sharp}=a_{ii}$,
for all $i$. By $(3ii)$ we have $\sum\limits_{i}a_{ii}\otimes
E_{ii}\in C_{2^n}-C_{2^n}$. Since $a_{ij}=a_{ij}'+ia_{ij}''$ for
some $a_{ij}'$, $a_{ij}''\in C_{2^n}-C_{2^n}$ we have
\begin{eqnarray*}a_{ij}\otimes E_{ij}+a_{ij}^{\sharp}\otimes
E_{ji}&=&(a_{ij}'+ia_{ij}'')\otimes E_{ij}+(a_{ij}'-ia_{ij}'')\otimes E_{ji}\\
&=&(a_{ij}'\otimes E_{ij}+a_{ij}'\otimes E_{ji})+(ia_{ij}''\otimes
E_{ij}-ia_{ij}''\otimes E_{ji})\\
&=&(E_{ii}+E_{ji})(a_{ij}'\otimes E_{ii}+a_{ij}'\otimes
E_{jj})(E_{ii}+E_{ij})\\
&-&(a_{ij}'\otimes E_{ii}+a_{ij}'\otimes E_{jj})\\
&+&(E_{ii}-iE_{ji})(a_{ij}''\otimes E_{ii}+a_{ij}''\otimes
E_{jj})(E_{ii}+iE_{ij})\\
&-&(a_{ij}''\otimes E_{ii}+a_{ij}''\otimes E_{jj})\in
C_{2^n}-C_{2^n}.
\end{eqnarray*}
Thus $A\in C_{2^n}-C_{2^n}$ and $A^{\sharp}=A$. Since for every
$x\in M_{2^n}(\mathcal{B})$ there exist unique $x_1=x_1^{\circ}$
and $x_2=x_2^{\circ}$ in $M_{2^n}(\mathcal{B})$, such that
$x=x_1+ix_2$, and unique $x_1'=x_1'^{\sharp}$ and
$x_2'=x_2'^{\sharp}$, such that $x=x_1'+ix_2'$, we have that
$x_1=x_1^{\sharp}=x_1'$, $x_2=x_2^{\sharp}=x_2'$ and involutions
$\sharp$ and $\circ$ coincide.
\end{proof}
\begin{lemma}\label{invandcone}
Involution $x\to x^\sharp$ is continuous on $\mathcal{B}_{\infty}$
and extends to the involution on
$\overline{\mathcal{B}}_{\infty}$. With respect to this involution
$\overline{C}\subseteq (\overline{\mathcal{B}}_{\infty})_{sa}$ and
$x^{\sharp}\overline{C}x\subseteq\overline{C}$ for every $x\in
\overline{\mathcal{B}}_{\infty}$.
\end{lemma}
\begin{proof}
Consider a convergent net $\{x_i\}\subseteq\mathcal{B}_{\infty}$
with the limit $x\in \mathcal{B}_{\infty}$. Decompose
$x_i=x_i'+ix_i''$ with $x_i', x_i'' \in C-C$. By (5), the nets
$\{x_i'\}$ and $\{x_i''\}$ are also convergent. Thus $x=a+ib$, where
$a=\lim x_i'\in \overline{C-C}$, $b=\lim x_i''\in \overline{C-C}$
and $\lim x_i^\sharp = a -i b$. Therefore the involution defined on
$\mathcal{B}_{\infty}$ can be extended by continuity to
$\overline{\mathcal{B}}_{\infty}$ by setting $x^\sharp=  a- ib$.

 Under this involution $\overline{C}\subseteq
(\overline{\mathcal{B}}_{\infty})_{sa}=\{x\in\overline{\mathcal{B}}_{\infty}:x=x^{\sharp}\}$.

 Let us show that  $x^{\sharp}cx\in\overline{C}$ for every
$x\in\overline{\mathcal{B}}_{\infty}$ and $c\in\overline{C}$. Take
firstly $c\in C_{2^n}$ and $x\in \mathcal{B}_{2^n}$. Then
$x=x_1+ix_2$ for some $x_1$, $x_2\in C_{2^n}-C_{2^n}$ and
\begin{gather*}
(x_1+ix_2)^{\sharp}c(x_1+ix_2)=(x_1-ix_2)c(x_1+ix_2)\\
= \frac{1}{2}\left(%
\begin{array}{cc}
  1 & 1\\
\end{array}%
\right)\left(%
\begin{array}{cc}
  -x_1 & -ix_2 \\
  ix_2 & x_1 \\
\end{array}%
\right)\left(%
\begin{array}{cc}
  c & 0 \\
  0 & c\\
\end{array}%
\right)\left(%
\begin{array}{cc}
  -x_1 & -ix_2 \\
  ix_2 & x_1 \\
\end{array}%
\right)\left(%
\begin{array}{c}
  1 \\
  1 \\
\end{array}%
\right)
\end{gather*}
By (3i), Lemma \ref{involution} and (3ii) $x^{\sharp}cx\in
C_{2^n}$.

Let now $c\in \overline{C}$ and
$x\in\overline{\mathcal{B}}_{\infty}$. Suppose that $c_i\rightarrow
c$ and $x_i\rightarrow x$, where $c_i\in C$, $x_i\in
\mathcal{B}_{\infty}$. We can assume that $c_i$, $x_i\in
B_{2^{n_i}}$. Then $x_i^{\sharp}c_ix_{i}\in C_{2^{n_i}}$ for all $i$
and since it is convergent we have $x^{\sharp}cx\in \overline{C}$.
\end{proof}

\begin{lemma}\label{ArchOrder}
The unit of $\overline{\mathcal{B}}_{\infty}$ is an Archimedean
order unit and
$(\overline{\mathcal{B}}_{\infty})_{sa}=\overline{C}-\overline{C}$.
\end{lemma}
\begin{proof}
Firstly let us show that $e_{\infty}$ is an order unit. Clearly,
$(\overline{\mathcal{B}}_{\infty})_{sa}=\overline{C-C}$. For every
$a\in \overline{C-C}$, there is a net $a_i\in
C_{2^{n_i}}-C_{2^{n_i}}$ convergent to $a$. Since
$\sup\limits_{i}\|a_i\|<\infty$ there exists $r_1>0$ such that
$r_1e_{n_i}-a_i\in C_{2^{n_i}}$, i.e. $r_1e_{\infty}-a_i\in C$.
Passing to the limit we get $r_1e_{\infty}-a\in \overline{C}$.
Replacing $a$ by $-a$ we can find $r_2>0$ such that
$r_2e_{\infty}+a\in\overline{C}$. If $r=\max(r_1,r_2)$ then
$re_{\infty}\pm a\in \overline{C}$. This proves that $e_{\infty}$ is
an order unit and that for all $a\in\overline{C-C}$ we have
$a=re_{\infty}-c$ for some $c\in\overline{C}$. Thus
$\overline{C-C}\in \overline{C}-\overline{C}$. The converse
inclusion, clearly, holds. Thus
$\overline{C-C}=\overline{C}-\overline{C}$.

If $x\in(\overline{\mathcal{B}}_{\infty})_{sa}$ such that for every
$r>0$ we have $r+x\in \overline{C}$ then $x\in\overline{C}$ since
$\overline{C}$ is closed. Hence $e_{\infty}$ is an Archimedean order
unit.
\end{proof}

\begin{lemma}\label{ker}  $\mathcal{B}_{\infty}\cap \overline{C}=C$.
\end{lemma}
\begin{proof}
Denote by $\mathcal{D}=\underrightarrow{\lim} M_{2^n}(B(H))$ the
$C\sp*$-algebra inductive limit corresponding to the inductive
system $\phi_n$ and denote
$\phi_{n,m}=\phi_{m-1}\circ\ldots\circ\phi_n:M_{2^n}(B(H))\rightarrow
M_{2^m}(B(H))$. For $n<m$ we identify $M_{2^{m-n}}(M_{2^n}(B(H)))$
with  $M_{2^m}(B(H))$ by omitting superfluous parentheses in a block
matrix $B=[B_{ij}]_{ij}$ with $B_{ij}\in M_{2^n}(B(H))$.

Denote by
 $P_{n,m}$  the operator  $diag (I, 0,\ldots, 0) \in M_{2^{m-n}}(M_{2^n}(B(H)))$
 and set $V_{n,m} = \sum_{k=1}^{2^{m-n}}
 E_{k,k-1}$.  Here $I$ is the identity matrix in $M_{2^n}(B(H))$ and
 $E_{k,k-1}$ is $2^n \times 2^n$ block matrix with identity operator at $(k,k-1)$-entry
 and all other entries being zero. Define an operator $\psi_{n,m}([B_{ij}]) = diag (B_{11}, \ldots,
 B_{11})$. It is easy to see that
 $$\psi_{n,m}([B_{ij}]) = \sum_{k=0}^{2^{m-n}-1} (V_{n,m}^k P_{n,m}) B
 (V_{n,m}^k
 P_{n,m})^*.$$ Hence by $(3ii)$
 \begin{gather}\label{inc}
 \psi_{n,m} (C_{2^m}) \subseteq \phi(C_{2^n})\subseteq C_{2^m}.
 \end{gather}
Clearly, $\psi_{n,m}$ is a linear contraction and
$$\psi_{n,m+k}\circ \phi_{m,m+k}=\phi_{m,m+k}\circ\psi_{n,m}$$
Hence there is a well defined contraction
$\psi_n=\lim\limits_{m}\psi_{n,m}:\mathcal{D}\rightarrow
\mathcal{D}$ such that
$$\psi_n|_{M_{2^n}(B(H))}=id_{M_{2^n}(B(H))},$$ where $M_{2^n}(B(H))$
is considered as a subalgebra in $\mathcal{D}$. Clearly,
$\psi_n(\overline{\mathcal{B}}_{\infty})\subseteq
\overline{\mathcal{B}}_{\infty}$ and
$\psi_n|_{\mathcal{B}_{2^n}}=id$. Consider $C$ and $C_{2^n}$ as
subalgebras in $\mathcal{B}_{\infty}$, by~(\ref{inc}) we have
$\psi_n:C\to C_{2^n}$.

To prove that $\mathcal{B}_{\infty}\cap \overline{C} =C$ take $c\in
\mathcal{B}_{\infty}\cap \overline{C}$. Then there is a net $c_j$ in
$C$ such that $\|c_j-c\|\to 0$. Since $c\in \mathcal{B}_{\infty}$,
$c\in \mathcal{B}_{2^n}$ for some $n$, and consequently
$\psi_n(c)=c$. Thus $$ \| \psi_n(c_j) - c \| = \| \psi_n(c_j-c)\|\le
\| c_j-c \|. $$ Hence $\psi_n(c_j)\to c$. But $\psi_n(c_j)\in
C_{2^n}$ and the latter is closed. Thus $c\in C$. The converse
inclusion is obvious.
\end{proof}
\begin{remark}\label{rem1}
Note that for every $x\in \mathcal{D}$
\begin{gather}\label{cutdiag}
\lim_n \psi_n(x) = x.
\end{gather}
Indeed,  for every $\varepsilon > 0$ there is $x\in M_{2^n}(B(H))$
such that $\| x-x_n \|< \varepsilon$. Since $\psi_n$ is a
contraction and $\psi_n(x_n) = x_n$ we have
\begin{eqnarray*}
\norm{\psi_n(x) - x} &\le& \norm{\psi_n(x)-x_n}+ \norm{x_n-x}\\
&=&\norm{\psi_n(x-x_n)}+\norm{x_n-x}\le 2 \varepsilon.
\end{eqnarray*}
Since $x_n\in M_{2^n}(B(H))$ also belong to $M_{2^m}(B(H))$ for all
$m\ge n$, we have that $\norm{\psi_m(x) - x}\le 2 \varepsilon$. Thus
$\lim\limits_{n} \psi_n(x) = x$.
\end{remark}

\noindent {\bf Proof of  Theorem~\ref{main}. } By Lemma
\ref{invandcone} and \ref{ArchOrder} the cone $\overline{C}$ and the
unit $e_{\infty}$ satisfies all assumptions of Theorem
\ref{onecone}. Thus there is a homomorphism
$\tau:\overline{\mathcal{B}}_{\infty}\rightarrow B(\widetilde{H})$
such that $\tau(a^{\sharp})=\tau(a)^*$ for all
$a\in\overline{\mathcal{B}}_{\infty}$. Since the image of $\tau$ is
a $*$-subalgebra of $B(\widetilde{H})$  we have that $\tau$ is
bounded by ~\cite[(23.11), p. 81]{DoranBelfi}. The arguments at the
end of the proof of Theorem \ref{operatorstar} show that the
restriction of $\tau$ to ${\mathcal{B}_{2^n}}$ is unitary equivalent
to the $2^n$-amplification of $\tau|_\mathcal{B}$. Thus
$\tau|_\mathcal{B}$ is completely bounded.

Let us prove that $\ker(\tau)=\{0\}$. By Theorem
\ref{operatorstar}.\ref{onecone3} it is sufficient to show that
$\overline{C}\cap(-\overline{C})=0$. If $c, d\in \overline{C}$ such
that $c+d=0$ then $c=d=0$. Indeed, for every $n\ge 1$, $\psi_n(c)
+\psi_n(d) = 0$. By Lemma \ref{ker}, we have
$$ \psi_n(\overline{C}) \subseteq \overline{C}\cap \mathcal{B}_{2^n} = C_{2^n}.$$
Therefore $\psi_n(c)$, $\psi_n(d)\in C_{2^n}$. Hence $\psi_n(c)=
-\psi_n(d)\in C_{2^n} \cap (- C_{2^n})$ and, consequently,
$\psi_n(c)= \psi_n(d) = 0$. Since $\norm{\psi_n(c)-c}\to 0$ and
$\norm{\psi_n(d)-d}\to 0$ by Remark \ref{rem1}, we have that
$c=d=0$. If $x\in \overline{C}\cap(-\overline{C})$ then $x+(-x)=0$,
$x, -x\in \overline{C}$ and $x=0$. Thus $\tau$ is injective.

We will show that the image of $\tau$ is closed if one of the
conditions $(1)$ or $(2)$ of the statement holds.

Assume firstly that operator algebra $\mathcal{B}$ satisfies the
first condition. Since $\tau(\overline{\mathcal{B}}_{\infty}) =
\tau(\overline{C}) -\tau(\overline{C}) +i
(\tau(\overline{C})-\tau(\overline{C}))$ and $\tau(\overline{C})$ is
exactly the set of positive operators in the image of $\tau$, it is
suffices to prove that $\tau(\overline{C})$ is closed. By Theorem
\ref{onecone}.\ref{onecone3}, for self-adjoint (under involution
$\sharp$) $x\in \overline{\mathcal{B}}_{\infty}$ we have
$$ \|\tau(x)\|_{B(\widetilde{H})} = \inf\{ r>0: re_{\infty}\pm
x\in \overline{C}\}.
$$
If $\tau(c_\alpha)\in \tau(C)$ is a Cauchy net in $B(\widetilde{H})$
then for every $\varepsilon>0$ there is $\gamma$ such that
$\varepsilon\pm (c_\alpha-c_\beta)\in \overline{C}$ when  $\alpha\ge
\gamma$ and $\beta\ge \gamma$. Since $\overline{C}\cap
\mathcal{B}_{\infty} = C$, $\varepsilon\pm (c_\alpha-c_\beta)\in C$.
Denote $c_{\alpha \beta} = \varepsilon+ (c_\alpha-c_\beta)$ and
$d_{\alpha \beta} = \varepsilon- (c_\alpha-c_\beta)$. The set of
pairs $(\alpha,\beta)$ is directed if $(\alpha,\beta)\ge
(\alpha_1,\beta_1)$ iff $\alpha\ge \alpha_1$ and $\beta\ge \beta_1$.
Since $c_{\alpha \beta} + d_{\alpha \beta} = 2 \varepsilon$ this net
converges to zero in  the norm of $\overline{\mathcal{B}}_{\infty}$.
Thus by assumption  $4$ in the definition of $*$-admissible sequence
of cones, $\|c_{\alpha \beta}\|_{\overline{\mathcal{B}}_{\infty}}\to
0$. This implies that $c_\alpha$ is a Cauchy net  in
$\overline{\mathcal{B}}_{\infty}$. Let $c=\lim c_\alpha$. Clearly,
$c\in \overline{C}$. Since $\tau$ is continuous $\| \tau(c_\alpha) -
\tau(c) \|_{\overline{\mathcal{B}}_{\infty}} \to 0$. Hence the
closure $\overline{\tau(C)}$ is contained in $\tau(\overline{C})$.
By continuity of $\tau$ we have $\tau(\overline{C})\subseteq
\overline{\tau(C)}$. Hence $\tau(\overline{C})= \overline{\tau(C)}$,
$\tau(\overline{C})$ is closed.

 Let now
$\mathcal{B}$ satisfy condition $(2)$ of the Theorem. Then for every
$x\in \overline{\mathcal{B}}_{\infty}$ we have $\| x^\sharp x \| \ge
\alpha \| x \| \| x^\sharp \|$.
 By~\cite[theorem 34.3]{DoranBelfi}
 $\overline{\mathcal{B}}_{\infty}$ admits an equivalent $C^*$-norm
 $\abs{\cdot}$. Since  $\tau$ is a faithful $*$-representation of the
 $C^*$-algebra  $(\overline{\mathcal{B}}_{\infty}, \abs{\cdot})$
 it is isometric.  Therefore $\tau(\overline{\mathcal{B}}_{\infty})$
 is closed.

Let us show that
$(\tau|_\mathcal{B})^{-1}:\tau(\mathcal{B})\rightarrow \mathcal{B}$
is completely bounded. The image
$\mathcal{A}=\tau(\overline{\mathcal{B}}_{\infty})$ is a
$C^*$-algebra in $B(\widetilde{H})$ isomorphic to
$\overline{\mathcal{B}}_{\infty}$. By Johnson's theorem
(see~\cite{Jo}), two Banach algebra norms on a semi-simple algebra
are equivalent, hence, $\tau^{-1}:\mathcal{A}\to
\overline{\mathcal{B}}_{\infty}$ is bounded homomorphism, say
$\|\tau^{-1}\| = R$. Let us show that
$\|(\tau|_\mathcal{B})^{-1}\|_{cb} = R$. Since
$$\tau|_{\mathcal{B}_{2^n}} = U_n (\tau|_{\mathcal{B}} \otimes id_{M_{2^n}} )U_n^*,
$$ for some unitary   $U_n: K\otimes \mathbb{C}^{2^n}\to
\widetilde{H}$  we have for any $B= [b_{ij}]\in
M_{2^n}(\mathcal{B})$
\begin{eqnarray*} \| \sum b_{ij}\otimes E_{ij} \| &\le& R \|\tau(\sum b_{ij}\otimes E_{ij})\|  \\  &=& R \| U_n( \sum
\tau(b_{ij}) \otimes E_{ij}  )U_n^* \|\\&=& R \| \sum \tau(b_{ij})
\otimes E_{ij} \|. \end{eqnarray*}This is equivalent to $$ \| \sum
\tau^{-1} (b_{ij})\otimes e_{ij} \| \le R \| \sum b_{ij}\otimes
E_{ij} \|,$$ hence $\| (\tau^{-1})^{2^n} (B) \|\le R \| B \|.$ This
proves that $\|(\tau|_\mathcal{B})^{-1}\|_{cb} = R$.

 The converse
statement evidently holds with $*$-admissible sequence of cones
given
 by $(\tau^{(n)})^{-1}(M_n(\mathcal{A})^+)$.$\Box$

Conditions (1) and (2) were used to prove that the image of
isomorphism $\tau$ is closed. The natural question one can ask is
wether there exists an operator algebra $\mathcal{B}$ and
isomorphism $\rho:\mathcal{B}\to B(H)$ with  non-closed self-adjoint
image. The following example gives the affirmative answer.
\begin{example}
Consider the algebra $\mathcal{B} = C^1([0,1])$ as an operator
algebra in $C\sp*$-algebra $\bigoplus\limits_{q\in \mathbb{Q}}
M_2(C([0,1]))$ via inclusion
$$f(\cdot)\mapsto \oplus_{q\in \mathbb{Q}}\left(
                    \begin{array}{cc}
                      f(q) & f'(q) \\
                      0 & f(q) \\
                    \end{array}
                  \right).
$$ The induced norm $$\norm{f} = \sup\limits_{q\in \mathbb{Q}} \left[ \frac{1}{2}(
2\abs{f(q)}^2+ \abs{f'(q)}^2 + \abs{f'(q)}\sqrt{4
\abs{f(q)}^2+\abs{f'(q)}^2}) \right]^{\frac{1}{2}}$$ satisfies the
inequality $ \norm{f}\ge
\frac{1}{\sqrt{2}}\max\{\norm{f}_\infty,\norm{f'}_\infty\}\ge
\frac{1}{2\sqrt{2}}\norm{f}_1$ where $\norm{f}_1 =
\norm{f}_\infty+\norm{f'}_\infty$ is the standard Banach norm on
$C^1([0,1])$. Thus $\mathcal{B}$ is a closed operator algebra with
isometric involution $f^\sharp (x) = \overline{f(x)}$, ($x\in
[0,1]$). The identity map $C^1([0,1]) \to C([0,1])$, $f\mapsto f$ is
a $*$-isomorphism of $\mathcal{B}$ into $C\sp*$-algebra with
non-closed self-adjoint image.
\end{example}

\section{Operator Algebra associated with Kadison's similarity problem. }

In 1955 R. Kadison raised the following problem. Is any bounded
homomorphism $\pi$ of a $C^*$-algebra $\mathcal{A}$ into $B(H)$
similar to a $*$-representation? The similarity above means that
there exists invertible operator $S\in B(H)$ such that $x\to
S^{-1}\pi(x) S$ is a $*$-representation of $\mathcal{A}$.

The following criterion due to Haagerup (see~\cite{Haagerup}) is
widely used in reformulations of Kadison's problem:  non-degenerate
homomorphism $\pi$ is similar to a $*$-representation iff $\pi$ is
completely bounded. Moreover the similarity $S$ can be chosen in
such a way that $\|S^{-1}\| \|S\| = \|\pi\|_{cb}$.

The affirmative answer to the Kadison's problem is obtained  in many
important cases. In particular, for nuclear $\mathcal{A}$,
 $\pi$ is automatically completely bounded with $\|\pi\|_{cb}\le \|\pi\|^2$ (see~\cite{Bunce}).

About recent state of the problem we refer the reader to
\cite{Pisier, Paulsen}.

We can associate an operator algebra $\pi(B)$ to every bounded
injective  homomorphism $\pi$ of a $C^*$-algebra  $\mathcal{A}$. The
fact that $\pi(B)$ is closed can be seen by restricting $\pi$ to a
nuclear $C^*$-algebra $C^*(x^*x)$. This restriction is similar to
$*$-homomorphism for every $x\in \mathcal{A}$ which gives the
estimate $\| x \|\le \|\pi\|^3 \|\pi(x)\|$ (for details
 see~\cite[p. 4]{pitts}). Denote $C_n= \pi^{(n)}(
M_n(\mathcal{A})^+)$.

Let $J$ be an involution in $B(H)$, i.e. self-adjoint operator such
that  $J^2 = I$. Clearly, $J$ is also a unitary operator.  A
representation $\pi: \mathcal{A} \to B(H)$ of a $*$-algebra
$\mathcal{A}$ is called $J$-symmetric if $\pi(a^*) = J \pi(a)^* J$.
Such representations are natural analogs of $*$-representations for
Krein space with indefinite metric $[x,y] = \langle  J x,y \rangle
$.

We will need the following observation due to V.
Shulman~\cite{Shulman} (see also ~\cite[lemma 9.3,
p.131]{KissinShulman}). If $\pi$ is an arbitrary representation of
$\mathcal{A}$ in $B(H)$ then the representation $\rho: \mathcal{A}
\to B(H\oplus H)$, $a \mapsto \pi(a) \oplus \pi(a^*)^*$ is
$J$-symmetric with $J(x\oplus y) = y \oplus x$ and representation
$\pi$ is a restriction $\rho|_{K\oplus \{0\}}$. Moreover, if $\rho$
is similar to $*$-representation then so is $\pi$. Clearly the
converse is also true, thus $\pi$ and $\rho$ are simultaneously
similar to $*$-representations or not. In sequel for an operator
algebra $\mathcal{D} \in B(H)$ we denote by
$\overline{\underrightarrow{\lim} M_{2^n}(\mathcal{D})}$ the closure
of the algebraic direct limit of  of $M_{2^n}(\mathcal{D})$ in the
$C\sp*$-algebra direct limit of inductive system  $M_{2^n}(B(H))$
with standard inclusions $x \to \left(
                    \begin{array}{cc}
                      x & 0 \\
                      0 & x \\
                    \end{array}
                  \right)$.

\begin{theorem}\label{kadison}
Let $\pi:\mathcal{A}\to B(H)$ be a  bounded unital $J$-symmmetric
injective homomorphism of a $C^*$-algebra $\mathcal{A}$ and let
$\mathcal{B}=\pi(\mathcal{A})$. Then $\pi^{-1}$ is a completely
bounded homomorphism. Its extension $\widetilde{\pi^{-1}}$ to the
homomorphism between the inductive limits
$\overline{\mathcal{B}}_{\infty} = \overline{\underrightarrow{\lim}
M_{2^n}(\mathcal{B})}$ and $\overline{\mathcal{A}}_{\infty}  =
\overline{ \underrightarrow{\lim} M_{2^n}(\mathcal{A})}$ is
injective.
\end{theorem}
\begin{proof}
  Let us show that $\{C_n\}_{n\ge 1}$ is a
$*$-admissible sequence of cones.  It is routine to verify that
conditions (1)-(3) in the definition of $*$-admissible cones are
satisfied for $\{C_n\}$. To see that condition $(4)$  also holds
take $B\in C_{n} - C_{n}$ and denote  $r = \norm{B}$. Let $D \in
M_{n}(A)_{sa}$ be such that
 $B = \pi^{(n)}(D)$. Since $\pi^{(n)}: M_n(\mathcal{A}) \to M_n(\mathcal{B})$
 is algebraic isomorphism it preserves spectra  $\sigma_{ M_n(\mathcal{A})}(x) =
 \sigma_{M_n(\mathcal{B})}(\pi^{(n)}(x))$. Since  the  spectral radius
$\spr(B)\le r$ we have  $\spr(D)\le r$. Hence  $r e_{n} + D \in
M_{n} (A)^+$ because $D$ is self-adjoint. Applying $\pi^{(n)}$ we
get $r e_{n} + B\in C_{n}$ which proves  condition (4).

Since $\pi$ is $J$-symmetric $$\|\pi^{(n)}(a)\| =\|(J\otimes E_n)
\pi^{(n)}(a)^* (J\otimes E_n) \|= \|\pi^{(n)}(a^*)\|$$ for every
$a\in M_n(\mathcal{A})$, and
 \begin{eqnarray*} \| \pi^{(n)}(h_1) \| &\le& 1/2 ( \|
\pi^{(n)}(h_1) + i \pi^{(n)}(h_2) \| + \| \pi^{(n)}(h_1) - i
\pi^{(n)}(h_2) \|) \\ &=&  \| \pi^{(n)}(h_1) + i \pi^{(n)}(h_2)
\|
\end{eqnarray*} for all $h_1, h_2 \in C_n-C_n$.
Thus condition (5) is satisfied and $\{C_n\}$ is $*$-admissible. By
Theorem~\ref{main}, there is an injective bounded  homomorphism
$\tau: \overline{\mathcal{B}}_\infty \to B(\widetilde{H})$ such that
its restriction to $\mathcal{B}$ is completely bounded,
$\tau(b^\sharp)=\tau(b)^*$ and  $\tau_n(C_n)=
\tau_n(M_n(\mathcal{B}))^+$.

Denote $\rho = \tau \circ \pi :\mathcal{A} \to B(\widetilde{H})$.
Since $\rho$ is a positive  homomorphism, it is a
$*$-representation. Moreover, $\ker \rho = \{0\}$ because both $\pi$
and $\tau$ are injective. Therefore $\rho^{-1}$ is $*$-isomorphism.
Since $\tau: \mathcal{B}\to B(\widetilde{H})$ extends to an
injective homomorphism of inductive limit
$\overline{\mathcal{B}}_\infty$ and $\rho^{-1}$ is completely
isometric, we have that  $\pi^{-1} = \rho^{-1}\circ \tau$ extends to
injective homomorphism of $\overline{\mathcal{B}}_\infty$. It is
also clear that $\pi^{-1}$ is completely bounded as a superposition
of two completely bounded maps.
\end{proof}

\begin{remark}
The first statement of Theorem~\ref{kadison} can be deduced also
 from~\cite[Theorem 2.6]{pitts}.
\end{remark}

\begin{remark}
Note that condition (1) and (2) in Theorem~\ref{main} for cones
$C_n$ from the proof of  Theorem~\ref{kadison} is obviously
equivalent to $\pi$ being completely bounded.
\end{remark}

\begin{center}
{\bf Acknowledgments.}
\end{center}

The authors wish to express their thanks to Victor Shulman for
helpful comments and providing the reference \cite{Shulman}.

The work was written when the second author was visiting Chalmers
University of Technology in G\"oteborg, Sweden. The second author
was supported by the Swedish Institute.

\end{document}